\documentclass[12pt]{article}
\usepackage{hyperref}
\usepackage{cmap}                        
\usepackage[cp1251]{inputenc}            
\usepackage[english]{babel}
\usepackage[left=1in,right=1in,top=1in,bottom=1.5in]{geometry} 
\usepackage{amssymb,amsmath, amsthm, amscd,ifthen}
\usepackage{graphicx}

\newcommand{\R}{{\mathbb R}}
\newcommand{\N}{{\mathbb N}}

\newtheorem{opr}{Definition}
\newtheorem{thm}{Theorem}
\newtheorem{lem}[thm]{Lemma}
\newtheorem{cor}[thm]{Corollary}
\date{}

\newtheorem{prop}[thm]{Proposition}

\newtheorem{defn}{Definition}

\title{Erd\H os-Ko-Rado theorem for $\{0,\pm 1\}$-vectors}
\author{Peter Frankl, Andrey Kupavskii\footnote{Moscow Institute of Physics and Technology, Ecole Polytechnique F\'ed\'erale de Lausanne; Email: {\tt kupavskii@yandex.ru} \ \ Research supported by the grant RNF~16-11-10014.}}

\date{}

\begin{document}
\maketitle
\begin{abstract}The main object of this paper is to determine the maximum number of $\{0,\pm 1\}$-vectors subject to the following condition. All vectors have length $n$, exactly $k$ of the coordinates are $+1$ and one is $-1$, $n \geq 2k$. Moreover, there are no two vectors whose scalar product equals the possible minimum, $-2$. Thus, this problem may be seen as an extension of the classical Erd\H os-Ko-Rado theorem. Rather surprisingly there is a phase transition in the behaviour of the maximum at $n=k^2$. Nevertheless, our solution is complete. The main tools are from extremal set theory and some of them might be of independent interest.
\end{abstract}

\section{Introduction}
Let $[n] = \{1,\ldots, n\}$. Then any subset of the power set $2^{[n]}$ is called a family of subsets, or family for short. Another way of looking at families is by associating with a set $R\subset [n]$ its characteristic vector $v(F) = (x_1,\ldots, x_n)$ with $x_i = 1$ for $i\in F$ and $x_i = 0$ for $i\notin F$.

This association of a family $\mathcal F$ with a family of vectors $\mathcal V = \{v(F): F\in \mathcal F\}$ provides a fruitful connection between some geometric problems concerning $\R^n$ and families of subsets with restrictions on sizes of pairwise intersections. The by now classic result of Frankl and Wilson \cite{FW} is a good example. Kahn and Kalai \cite{KK} gave it a twist to deduce counterexamples to the famous Borsuk conjecture.

Raigorodskii \cite{Rai1} succeeded in improving the bounds in geometric applications by enlarging the scope of vectors from $\{0,1\}$-vectors to $\{0,\pm 1\}$-vectors. Some further applications of $\{0,\pm 1\}$-vectors in such questions one may find in \cite{ChR}, \cite{K}, \cite{MR}, \cite{PR}, \cite{Rai2}, \cite{Rai3}.

Extremal problems for $\{0,\pm 1\}$-vectors were considered before (cf. \cite{DF}), but no systematic investigation happened so far. The aim of the present paper is to consider extending the classical Erd\H os-Ko-Rado Theorem \cite{EKR} to this setting. This direction of research was proposed to us by Raigorodskii. Before stating the main results, let us introduce some definitions.

\begin{opr}\label{def1} For $0\le l,k<n$ define $\mathcal V(n,k,l) \subset \R^n$ to be the set of all $\{0,\pm 1\}$-vectors having exactly $k$ coordinates equal $+1$ and $l$ coordinates equal $-1$.
 \end{opr}
 We tacitly assume that $n\ge k+l$ and also $k\ge l$. Indeed, for $l>k$ one can replace a family $\mathcal V$ of vectors by $-\mathcal V = \{-v:v\in \mathcal V\}$ and have the number of $+1$'s prevail. Note that
\begin{equation}\label{eq00}
|\mathcal V(n,k,l)| = {n\choose k}{n-k\choose l}.
\end{equation}
With this notation families of $k$-sets are just subsets of $\mathcal V(n,k,0)$.

For vectors $v,w$ their \textit{scalar product} is denoted by $\langle v,w\rangle$. If both $v$ and $w$ are $\{0,1\}$-vectors, then their scalar product is non-negative with $\langle v(F),v(G)\rangle = 0$ iff $F\cap G = \emptyset$.

\begin{opr}\label{def2} A family $\mathcal F\subset 2^{[n]}$ is called \textit{intersecting}, if $F\cap G\ne \emptyset$ holds for all $F,G\in \mathcal F$. \end{opr}
For completeness let us state the Erd\H os-Ko-Rado Theorem.

\begin{thm}[\cite{EKR}]\label{thmekr} Suppose that $n\ge 2k>0$ and $\mathcal F\subset {[n]\choose k}$ is intersecting. Then
\begin{equation}\label{eq01} |\mathcal F|\le {n-1\choose k-1}.\end{equation}
\end{thm}

For $n\ge 2k$ the minimum possible scalar product in $\mathcal V(n,k,l)$ is $-2l$. Two vectors achieve this iff the $-1$ coordinates in each of them match $+1$ coordinates in the other, and no two $+1$-coordinates have the same position.

\begin{opr}\label{def3} A family $\mathcal V$ of vectors is called \textit{intersecting}, if the scalar product of any two vectors in $\mathcal V$ exceeds the minimum of the scalar product in $\mathcal V(n,k,l)$.
\end{opr}
By analogy with the Erd\H os-Ko-Rado Theorem define
$$m(n,k,l) = \{\max|\mathcal V|: \mathcal V\subset \mathcal V(n,k,l), \mathcal V\text{ is intersecting}\}.$$

 With this terminology the Erd\H os-Ko-Rado Theorem can be stated as
$$m(n,k,0) = {n-1\choose k-1}\text{\ \ \ for \ } n\ge 2k.$$

The present paper is mostly devoted to the complete determination of $m(n,k,1)$. The surprising fact is that the situation is very different from the case $l = 0$, the Erd\H os-Ko-Rado Theorem. Namely, for $n>k^2$ the Erd\H os-Ko-Rado-type construction is no longer optimal.

\begin{opr}\label{def4} For $n\ge 2k$ define $\mathcal E(n,k,1) = \{(x_1,\ldots, x_n)\in \mathcal V(n,k,1): x_1 = 1\}$.\end{opr}
Clearly, $\mathcal E(n,k,1)$ is intersecting with
\begin{equation}\label{eq02}|\mathcal E(n,k,1) = k{n-1\choose k}.\end{equation}
\begin{thm}\label{thmsmall} For $2k\le n\le k^2$
\begin{equation}\label{eq03}m(n,k,1) = k{n-1\choose k} \text{\ \ \ \ holds.}\end{equation}
\end{thm}

The proof of (\ref{eq03}) is much simpler in the range $2k\le n\le 3k-1$. It can be done using Katona's Circle Method (cf. \cite{Ka1}). We are going to present this case in Section \ref{sec4}. However, for $n\geq 3k$ the proof of (\ref{eq03}) is much harder and much more technical. It is postponed till Section \ref{sec6}.\\

What happens for $n>k^2$?

Suppose that $\mathcal V\subset \mathcal V(n,k,1)$ is intersecting. Adding an extra $0$ in the $n+1$'st coordinate makes it possible to consider $\mathcal V$ as a subset of $\mathcal V(n+1,k,1)$. Adding the ${n\choose k}$ vectors that have $-1$ in the $n+1$'st coordinate produce an intersecting family $\mathcal P(\mathcal V)\subset\mathcal V(n+1,k,1)$. This implies the following inequality:
\begin{equation}\label{eq04}
m(n+1,k,1)\ge m(n,k,1)+{n\choose k}.
\end{equation}
As a counterpart for (\ref{eq04}) we prove

\begin{thm}\label{thmbig} For $n\ge k^2$ one has
\begin{equation}\label{eq05} m(n+1,k,1)=m(n,k,1)+{n\choose k}.
\end{equation}\end{thm}

As an immediate consequence we have the following.
\begin{cor} For $n> k^2$ one has
$$m(n,k,1) = m(k^2,k,1)+{k^2\choose k}+{k^2+1\choose k}+\ldots +{n-1\choose k}.$$
\end{cor}

We remark that in the case of $l>1$ the problem becomes much harder. We have managed to solve it completely for $2k\le n\le 3k-l$ and asymptotically (for $n$ large with respect to $k$) in \cite{FK10}. However, finding the exact value of $m(n,k,l)$ for $n>3k-l$ is still wide open.

The paper is organized as follows. In the next section we present a brief summary of the paper, containing the ideas and the logic of the proofs. In Section \ref{sec31} we discuss the constructions of intersecting families. In Section \ref{sec3} we summarize all the necessary material from extremal set theory, in particular, concerning shadows, shifting and cross-intersecting families, and prove several auxiliary statements. In Section \ref{sec4} we prove (\ref{eq03}) in the case $2k\le n\le 3k-1$. In Section \ref{sec5} we prove Theorem \ref{thmbig}. In Section \ref{sec6} we prove (\ref{eq03}) in the case $3k\le n\le k^2$.

Throughout the paper we assume that $n\ge 2k$ and that $k>1$.

\section{Summary}
This section is meant to summarize the structure of the proof and often refers to the statements in the proof. Therefore, some parts of it may be difficult to follow before actually reading through some parts of the proof. Thus, rather then giving a rough idea for the reader beforehand, it is aimed to help readers to see through the (rather complicated) details when reading the proof.\\

First, for the whole proof we assume that the families are shifted, i.e., $1$'s normally appear before $0$'s and $-1$'s, and $0$'s appear before $-1$'s among the coordinate positions of vectors from our family. The precise definition of shifting and the proof of the fact that it preserves the property that the family is intersecting is discussed in Section \ref{sec321}. \\

We start with the sketch of the proof of Theorem \ref{thmbig}. The goal is to show that the number of the vectors containing $+1$ or $-1$ on the last coordinate is at most ${n\choose k}$. These vectors fall into two groups: $\mathcal A$, with $-1$ at the end, and $\mathcal B$, with $1$ at the end. We fix another coordinate $i<n+1$ and take the vectors from  $\mathcal A$, denoted $\mathcal A(i)$, that have 1 on $i$-th position, and vectors from $\mathcal B$, denoted $\mathcal B(-i)$, that have $-1$ on $i$-th position.

For each such $i$ we may treat the resulting families of vectors as two set families of $(k-1)$-sets, which are cross-intersecting. We remark that, due to shifting, $\mathcal B(-i)\subset \mathcal A(i)$ and, in particular, $\mathcal B(-i)$ is intersecting. Then we bound the expression $|\mathcal A(i)|+k|\mathcal B(-i)|$ from above using Theorem \ref{thmcross} from Section \ref{sec44}. Finally we average the result over all possible $i$, obtaining a desired bound on $|\mathcal A|+|\mathcal B|$. Note that the coefficient $k$ is needed so that we get the expression $|\mathcal A|+|\mathcal B|$ after averaging.\\

Next, we sketch the proof of Theorem \ref{thmsmall}. For $2k\le n\le 3k-1$ the  theorem is obtained via a direct application of Katona's circle method. The general form of this method is discussed in Section \ref{sec43}. The only trick is to  choose a good subset of vertices to which we can apply the method.\\

The case of $3k\le n\le k^2$ is the most difficult part of the proof. Again, the idea is to apply the argument with the averaging used for the proof of Theorem \ref{thmbig}. However, in this case there is a major complication. If one takes a look at Theorem \ref{thmcross} with given parameters (the size of the ground set is $n-1$, since $i$ and $n+1$ are not present, and the size of each set in $\mathcal A(i), \mathcal B(-i)$ is $k-1$), the maximum value of the expression $|\mathcal A(i)|+k|\mathcal B(-i)|$ is $(k+1){n-2\choose k-2}$ and is attained in the case when both are the trivial intersecting families.
 However, if we look at the Erd\H os-Ko-Rado-type family, which is expected to be maximal in this range, and the corresponding sets $\mathcal A(i), \mathcal B(i)$, then they are indeed the trivial intersecting families for all $i$ except for $i = 1$. In that case we have the other extreme: $\mathcal B(1)$ is empty and $\mathcal A(1) = {[n-1]\choose k-1}$. Therefore, if we apply the bound from Theorem \ref{thmcross} for all $i$ blindly, then after the averaging we get a worse bound, with the difference of $(k+1){n-2\choose k-2}-{n-1\choose k-1}$ from the size of the Erd\H os-Ko-Rado-type family.

The idea to circumvent it is as follows. If $\mathcal B(-1)$ is empty, then we are home. If not, then, due to the fact that the whole family is shifted, we may conclude that there is a relatively big set $I$, $I\subset [n]$, such that for all $i\in I$ there are  many sets in  $\mathcal B(-i)$  that do not contain element $\{1\}$.
The precise statement is Proposition \ref{prop4}. The next step is to use Corollary \ref{cordeg} of a theorem due to the first author from Section \ref{sec47}, which roughly states that we can bound non-trivially the size of the intersecting family from above, provided that we know that it is far from trivial intersecting family, that is, if there are many sets that do not contain the element with the biggest degree. By non-trivial we mean a bound that is smaller than the size of the trivial intersecting family.

The result of the manipulations presented in the previous paragraph is a non-trivial bound on the size of each $\mathcal B(-i)$, $i\in I$, provided that $\mathcal B(-1)$ is non-empty.  Finally, we bound the size of each $|\mathcal A(i)|+k|\mathcal B(-i)|$ using Theorem \ref{thmcross2} from  Section \ref{sec46}, which is a refined version of Theorem \ref{thmcross}. Using this bound, we apply the same averaging as in Theorem \ref{thmbig} and show that in all cases the size of $|\mathcal A|+|\mathcal B|$ is at most the size of $|\mathcal A|+|\mathcal B|$ when $\mathcal B(-1)$ is empty. Speaking very roughly, the sets in $\mathcal B(-1)$ force the sets $\mathcal B(-i)$ to be small, and thus force the whole sum $|\mathcal A|+|\mathcal B|$ to be small.
The large part of Section \ref{sec6} is devoted to the calculations that ensure that it is indeed the case.

Theorem \ref{thmcross2}, which gives very fine-grained bounds on $|\mathcal A(i)|+k|\mathcal B(-i)|$ depending on the size of $\mathcal B(-i)$, is itself one of the complicated parts of the proof. First, we do a detailed analysis of the maximal cross-intersecting families using and refining Kruskal-Katona Theorem in Section \ref{sec45}. The results of this section are interesting in their own right, as they provide a better understanding of the structure of cross-intersecting families (we refer to the papers \cite{KZ}, \cite{Kup17}, where these and some complementary ideas were developed). The language of truncated characteristic vectors, introduced in Section \ref{sec45} seems to be very convenient. Lemmas \ref{lem12}, \ref{lem13} allow us to reduce the wide array of different cross-intersecting families to a few, one of which is guaranteed to be maximum w.r.t. the expression we maximize. The proof of Theorem \ref{thmcross2} itself is a more technical counterpart.

\section{Comparing the constructions}\label{sec31}
To get some intuition for the problem, we start with the comparison of the constructions of intersecting families briefly discussed in the introduction.

The first intersecting family $\mathcal E(n,k,1)$ is the Erd\H os-Ko-Rado-type family, mentioned in the introduction, in which all the vectors have $1$ on the first position. We have $e(n,k,1):=|\mathcal E(n,k,1)| =  k{n-1\choose k}.$ Note that we have $v(n,k,1) = |\mathcal V(n,k,1)| = (k+1){n\choose k+1}.$ Therefore, we have $e(n,k,1)/v(n,k,1) = k/n$.

The second  family $\mathcal P(n,k,1)$ consists of all the vectors for which the last non-zero coordinate is $-1$. It is easy to see that this is indeed an intersecting family. We have $p(n,k,1):= |\mathcal P(n,k,1)| = {n\choose k+1}$. Therefore, we have $p(n,k,1)/v(n,k,1) = 1/(k+1)$.

\begin{prop}\label{prop1} The inequality $e(n+1,k,1)-e(n,k,1)\ge p(n+1,k,1)-p(n,k,1)$ holds iff $n\le k^2$. We have equality iff $n = k^2$.
\end{prop}
\begin{proof} The proof is a matter of simple calculations:
$$e(n+1,k,1)-e(n,k,1) = k{n\choose k}-k{n-1\choose k} = k{n-1\choose k-1},$$
$$p(n+1,k,1)-p(n,k,1) = {n+1\choose k+1}-{n\choose k+1} = {n\choose k} = \frac nk{n-1\choose k-1}.$$
\vskip-0.6cm \end{proof}


The second construction of intersecting families allows for the following generalization, described in the introduction. Assume we are given an intersecting family $\mathcal F\subset \mathcal V(n,k,1)$. We can construct an intersecting family $\mathcal {P(F)}\subset \mathcal V(n+1,k,1)$ in the following way:
$$\mathcal {P(F)} = \{(\mathbf v, 0):\mathbf v\in \mathcal F\}\cup \{\mathbf w = (w_1,\ldots, w_{n+1}): \mathbf w\in V(n+1,k,1), w_{n+1} = -1\}.$$

Since $\mathcal F$ is intersecting, $\mathcal {P(F)}$ is intersecting as well.  We have
\begin{equation}\label{eq1} |\mathcal {P(F)}|-|\mathcal F| = {n\choose k} = p(n+1,k,1)-p(n,k,1).\end{equation}

We denote by $\mathcal P^s(\mathcal F)$ the result of $s$ consecutive applications of operation $\mathcal P(*)$ to  the family $\mathcal F$. This gives us the following composite construction of an intersecting family $\mathcal C(n,k,1)\subset \mathcal V(n,k,1)$.

\begin{equation}\label{eq2}\mathcal C(n,k,1) =\begin{cases}\mathcal E(n,k,1), & \text{ if  $n\le k^2$};\\ \mathcal P^{n-k^2}\bigl(\mathcal E(k^2,k,1)\bigr), & \text{ if  $n> k^2$}.\end{cases}\end{equation}

We denote the cardinality of $\mathcal C(n,k,1)$ by $c(n,k,1)$. We have $c(n,k,1) = e(n,k,1)$ for $n\le k^2$ and, due to to (\ref{eq1}), $c(n,k,1) = e(k^2,k)-p(k^2,k)+p(n,k)$ for $n>k^2$.
 By Proposition~\ref{prop1}, $\mathcal C(n,k,1)$ is the biggest intersecting family among the ones discussed in this subsection. In what follows we prove that $\mathcal C(n,k,1)$ has maximum cardinality among intersecting families in $\mathcal V(n,k,1)$.\\

\textbf{Remark.} Due to the fact that equality is possible in Proposition \ref{prop1}, there is a slightly different intersecting family that has exactly the same cardinality as $\mathcal C(n,k,1)$. Its definition is almost the same, we only have to replace $k^2$ by $k^2+1$ in (\ref{eq2}).

\section{Auxiliaries from extremal set theory}\label{sec3}
In this section we present several auxiliary results and techniques that we'll use in the latter sections. Some of the results presented here are well-known, while the others appear to be new and may be of independent interest.
\subsection{Shifting}\label{sec321}
We start with \textit{shifting} (left compression). For a given  pair of indices $i<j\in [n]$ and a vector $\mathbf v = (v_1,\ldots, v_n)\in \R^n$ we define an $(i,j)$-shift $\mathbf v_{i,j}$ of $\mathbf v$ in the following way. If $v_i\ge v_j$, then $\mathbf v_{i,j} = \mathbf v$. If $v_i<v_j$, then $\mathbf v_{i,j} = (v_1,\ldots, v_{i-1},v_j,v_{i+1},\ldots,v_{j-1},v_i,v_{j+1},\ldots, v_n),$ that is, it is obtained from $\mathbf v$ by interchanging its $i$-th and $j$-th coordinate.

Next, we define an $(i,j)$-shift $\mathcal Q_{i,j}$ of $\mathcal Q$ for a finite system of vectors $\mathcal Q\subset \R^n$. We take a vector $\mathbf v \in \mathcal Q$ and replace it with $\mathbf v_{i,j}$, if $\mathbf v_{i,j}$ is not already in $\mathcal Q$. If it is, then we leave $\mathbf v$ in the system. Formally,

$$\mathcal Q_{i,j} = \{\mathbf v_{i,j}: \mathbf v\in \mathcal Q\}\cup \{\mathbf v: \mathbf v,\mathbf v_{i,j}\in \mathcal Q\}.$$

We call a system $\mathcal Q$ \textit{shifted}, if $\mathcal Q = \mathcal Q_{i,j}$ for all $i<j\in [n]$. Any  system of vectors may be made shifted by means of a finite number of $(i,j)$-shifts. Here is the crucial lemma concerning shifting:

\begin{lem}\label{lem2} For any  $\mathcal Q\subset \R^n$ and any $i<j\in [n]$ we have
$$\min \{\langle \mathbf v,\mathbf w\rangle: \mathbf v, \mathbf w\in \mathcal Q\}\le \min \{\langle \mathbf v',\mathbf w'\rangle: \mathbf v', \mathbf w'\in \mathcal Q_{i,j}\}.$$
\end{lem}
\begin{proof} Take any two vectors $\mathbf v = (v_1,\ldots, v_n),\mathbf w = (w_1,\ldots, w_n)\in \mathcal Q$. We denote by $\mathbf v', \mathbf w'$ the result of the $(i,j)$-shift in $\mathcal Q$ applied to $\mathbf v,\mathbf w$ (that is, for $\mathbf v$ we have $\mathbf v' = \mathbf v$ or $\mathbf v_{i,j}$, depending on whether $\mathbf v_{i,j}$ is in $\mathcal Q$ or not). If $\mathbf v' = \mathbf v$ and $\mathbf w' = \mathbf w$, then, obviously, $\langle \mathbf v',\mathbf w'\rangle = \langle \mathbf v,\mathbf w\rangle$. Moreover, we have the same if both $\mathbf v' \ne \mathbf v$ and $\mathbf w' \ne \mathbf w$. Therefore, the only nontrivial case we need to consider is when $\mathbf v' \ne \mathbf v$ and $\mathbf w' = \mathbf w$.

The reasons for $\mathbf v'$ being different from $\mathbf v$ are unambiguous: $v_i<v_j$ and $\mathbf v_{i,j}\notin \mathcal Q$. For $\mathbf w'$, however, there are two possible reasons not to be shifted. The first one is that $w_i\ge w_j$ and, thus, $\mathbf w = \mathbf w_{i,j}$. Then
$$\langle \mathbf v',\mathbf w'\rangle - \langle \mathbf v,\mathbf w\rangle = v_iw_j+v_jw_i-v_iw_i-v_jw_j = (v_j-v_i)(w_i-w_j)\ge 0.$$
The second possible reason is that $w_i<w_j$, but  $\mathbf w_{i,j}\in\mathcal Q$. Then $$\langle \mathbf v',\mathbf w'\rangle = \langle \mathbf v_{i,j},\mathbf w\rangle = \langle \mathbf v,\mathbf w_{i,j}\rangle.$$
The last scalar product is, in fact, between two vectors from $\mathcal Q$. Therefore, in all cases we have exhibited a pair of vectors from $\mathcal Q$, that have a scalar product smaller than or equal to $\langle \mathbf v',\mathbf w'\rangle$. \end{proof}

Applied to our case, the lemma states that, given an intersecting family of vectors, we may replace it with a shifted family of vectors, and the shifted family is intersecting as well.

\subsection{Shadows}\label{sec322}
Given a family $\mathcal F\subset {[n]\choose k}$, we define its shadow $\sigma(\mathcal F)\subset {[n]\choose k-1}$ as a family of all $(k-1)$-element sets that are contained in one of the sets from $\mathcal F$. More generally, if $l<k$, then the $l$-th shadow $\sigma^l(\mathcal F)$ is the set of all $(k-l)$-element sets that are contained in one of $F\in \mathcal F$. The famous Kruskal-Katona theorem \cite{Ka}, \cite{Kr} gives a sharp lower bound on the size of the shadow of $\mathcal F$ in terms of $k$ and $|\mathcal F|$.  We are going to discuss it in the forthcoming paragraphs.


However, we need an analogous relation for the set system and its shadow, but for sets of specific type. Fix a cyclic permutation $\pi$ of $[n]$. Consider the set $\mathcal U(\pi,k)\subset {[n]\choose k}$ of $n$ $k$-sets, each of which forms an interval in the permutation $\pi$. That is, they are composed of cyclically  consecutive elements in the permutation.
\begin{lem}\label{lemshad} For any set system $\mathcal F\subset \mathcal U(\pi, k)$ we have $|\sigma^l(\mathcal F)\cap \mathcal U(\pi,k-l)|\ge \min\{|\mathcal F|+l, n\}$.
\end{lem}
\begin{proof} It is clearly sufficient to prove that $|\sigma(\mathcal F)\cap \mathcal U(\pi,k-1)|\ge \min\{|\mathcal F|+1, n\}$. If $|\mathcal F| = |\mathcal U(\pi,k)| = n$, then $\sigma(\mathcal F) = \mathcal U(\pi,k-1)$ and the statement is obvious. Therefore, we may assume that $\mathcal U(\pi,k)\setminus \mathcal F\ne\emptyset$.

Split the family $\mathcal F$ into subfamilies $\mathcal F_1,\ldots, \mathcal F_s$, each of which form an ``interval'' (or a ``tight path''). That is, each $\mathcal F_i$ is a maximal sequence of different sets $F^1_i,\ldots,F^d_i\in \mathcal F$, in which each pair of consecutive sets intersect in a $(k-1)$-element set. Clearly, this is a partition of $\mathcal F$ into equivalence classes.
Moreover, sets from different subfamilies intersect in less than $k-1$ elements.  Therefore, $\sigma(\mathcal F) = \bigsqcup_i \sigma(\mathcal F_i)$.

For each subfamily we have $|\sigma(\mathcal F_i)\cap \mathcal U(\pi,k-1)| = |\mathcal F_i|+1$. This is due to the fact that each $F\in \mathcal F_i$ contains two $(k-1)$-element sets from $\mathcal U(\pi,k-1)$, while each $F'\in \sigma(\mathcal F_i)\cap \mathcal U(\pi,k-1)$ is contained in either one or two sets from $\mathcal F_i$, and there are exactly two sets that are contained in one set from $\mathcal F_i$. Informally speaking, these are the ``left shadow'' of $F^1_i$ and the ``right shadow'' of $F^d_i$. These two shadow sets are different, since $\mathcal U(\pi,k)\setminus \mathcal F_i\ne \emptyset$. Knowing the ``degrees'' of the sets, we get the desired equality by simple double counting.

Finally, putting the statements for different $i$ together, we get that $|\sigma(\mathcal F)\cap \mathcal U(\pi,k-1)|\ge |\mathcal F|+s\ge |\mathcal F|+1.$ Repeating the argument $l$ times yields the result.
\end{proof}

\subsection{General form of Katona's circle method}\label{sec43}
For this subsection only we adopt the language of graph theory. Consider a graph $G = (V,E)$, which is vertex-transitive. That is, the group $Aut(G)$ of automorphisms acts transitively on $V$. For a given vertex $v\in V$ we denote by $S_v$ the stabilizer of $v$ in $Aut(G)$, which is a subgroup of all automorphisms of $G$ that map $v$ to itself. A basic observation in group theory states that the size of the stabilizer is the same for all the vertices  of $G$. Indeed, if $v, w$ are two vertices of $G$ and $\sigma\in Aut(G)$ maps $v$ to $w$, then $S_v = \sigma^{-1} S_w \sigma$ and, therefore, $|S_w| = |S_v|$. Moreover, $|S_v| = |S_{vw}|,$ where $S_{vw}$ is the set of elements of $Aut(G)$ that maps $v$ into $w$. We have as well $|Aut(G)| = |G||S_v|,$ where $|G|$ is the number of vertices in $G$.

We remind the reader that $\alpha (G)$ is the independence number of $G$, that is, the maximum number of vertices that are pairwise non-adjacent. The following lemma is a special case of Lemma 1 from \cite{Ka2}.
\begin{lem}[Katona, \cite{Ka2}]\label{lem66} Let $G$ be a vertex-transitive graph. Let $H \subset G$ be a subgraph of $G$. Then $\alpha(G)\le \frac{\alpha(H)}{|H|}|G|.$
\end{lem}
\textbf{Remark. } We formulated the lemma for the independence number, since it meets our demands. However, an analogue of it may be formulated for some other graph characteristics.
\begin{proof} For any $\sigma \in Aut(G)$ we denote an induced subgraph of $G$ on the set of vertices $\sigma(V(H)) = \{\sigma(v): v\in V(H)\}$ by $\sigma(H)$. The proof of the lemma goes by simple double counting. Before doing the crucial double counting step, we remark that the union over all $\sigma \in Aut(G)$ of $\sigma(V(H))$ covers each vertex exactly $|H||S_v|$ times. Take any independent set $I$ in $G$.
\begin{equation}\label{eq66} |I| = \frac 1{|H||S_v|}\sum_{\sigma \in Aut(G)} |I\cap \sigma(H)|\le \frac 1{|H||S_v|}\sum_{\sigma \in Aut(G)} \alpha(H) = \frac{|Aut(G)|}{|S_v|}\frac{\alpha(H)}{|H|} = \frac{\alpha(H)}{|H|}|G|.\end{equation}
\end{proof}

There is a natural connection between this lemma and intersecting families, which goes via Kneser graphs. A \textit{Kneser graph} $KG_{n,k}$ is graph which set of vertices is ${[n]\choose k}$, and two vertices are adjacent iff the corresponding sets are disjoint. By definition the value of $\alpha(KG_{n,k})$ is the size of a maximum intersecting family in ${[n]\choose k}$. Using Lemma\label{lem66} for the so-called Katona's circle, it is not difficult to show that the independence number of $KG_{n,k}$ is equal to ${n-1\choose k-1}$, which is the statement of the Erd\H os-Ko-Rado theorem.

\subsection{An inequality for cross-intersecting families of sets}\label{sec44}
In this subsection we prove a theorem about cross-intersecting families that we need for the proof of Theorem \ref{thmbig}. We say that two families $\mathcal A,\ \mathcal B$ are {\it cross-intersecting}, if for any $B\in \mathcal B, A\in \mathcal A$ we have $B\cap A\ne \emptyset$.

\begin{thm}\label{thmcross} Let $n\ge 2k$,  $c\ge 1$.  Consider two cross-intersecting families $\mathcal A,\ \mathcal B$, where $\mathcal B\subset\mathcal A\subset {[n]\choose k}$.  Then
\begin{equation}\label{eqcross} |\mathcal A|+c|\mathcal B|\le \max\Bigl\{{n\choose k}, (c+1){n-1\choose k-1}\Bigr\}.
\end{equation}
\end{thm}
\textbf{Remark.} Informally speaking, the theorem states that the sum is maximized in one of the two cases: either $\mathcal B$ is empty and we may take $\mathcal A$ to be ${[n]\choose k}$, or when $\mathcal A = \mathcal B$, and each of them is a trivial intersecting family, that is, in which all sets contain a fixed element.
\begin{proof} The proof is an application of Katona's cyclic permutation method. The following proposition is the key step.
\begin{prop}\label{prop2} Fix a  cyclic permutation $\pi$ of $[n]$. Consider the set $\mathcal U(\pi,k)\subset {[n]\choose k}$ from Section \ref{sec322}. Consider two  subfamilies $\mathcal A(\pi)= \mathcal A\cap \mathcal U(\pi,k)$ and $\mathcal B(\pi) = \mathcal B\cap \mathcal U(\pi,k)$. Denote $a = |\mathcal A(\pi)|, b = |\mathcal B(\pi)|$. Then

\begin{equation}\label{eq3} a+cb\le \max \{n, (c+1)k\}.\end{equation}
\end{prop}
\begin{proof}
If the set $\mathcal B(\pi)$ is empty, then the statement is trivial, since $a\le |\mathcal U(\pi,k)| = n$. Henceforth we assume that $|\mathcal B(\pi)| = s>0$.  We pass to the complements of the sets from $\mathcal B(\pi)$, considering the set $\overline{\mathcal B}(\pi) = \{\bar B: B\in \mathcal B(\pi)\}.$ On the one hand, we know that for each $A\in \mathcal A(\pi)$ and $\bar B\in \overline{\mathcal B}(\pi)$ we have $A\nsubseteq \bar B$. In other words, $A\notin \sigma^{n-2k}\bigl(\overline{\mathcal B}(\pi)\bigr).$ On the other hand, by Lemma \ref{lemshad} we have $\bigl|\sigma^{n-2k}\bigl(\overline{\mathcal B}(\pi)\bigr)\bigr|\ge \min\{|\mathcal B(\pi)|+n-2k, n\} = n-2k+s,$ for if $\bigl|\sigma^{n-2k}\bigl(\overline{\mathcal B}(\pi)\bigr)\bigr| = n$, then $\mathcal A(\pi)$ and, consequently, $\mathcal B(\pi)$ is forced to be empty.

Combining these two facts, we get $|\mathcal A(\pi)| \le n - (n-2k+s) = 2k-s$. From the following chain  $s = |\mathcal B(\pi)|\le |\mathcal A(\pi)| \le 2k-s$ we conclude that $s\le k$.  Finally,
$$|\mathcal A(\pi)|+c|\mathcal B(\pi)|\le 2k-s+cs = 2k+(c-1)s\le (c+1)k.$$
\vskip-0.5cm\end{proof}

Knowing Proposition \ref{prop2}, the rest of the proof of the lemma is a standard double-counting argument, which was, in fact, carried out in the proof of Lemma \ref{lem66}. We take the circle $U(\pi,k)$ as a subgraph $H$ from Lemma \ref{lem66}. In parallel to (\ref{eq66}), we get
$$|\mathcal A| +c |\mathcal B|\le \max \{n, (c+1)k\}\frac 1 n {n\choose k} =  \max\Bigl\{{n\choose k}, (c+1){n-1\choose k-1}\Bigr\}.$$
\vskip-0.5cm
\end{proof}

\subsection{Analysis of the Kruskal-Katona's Theorem}\label{sec45}
We note that the ideas of this section have been applied in the context of intersecting families, see \cite{Kup17}.

 For $i\le j$ denote $[i,j] = \{i,i+1,\ldots, j\}$. We introduce a lexicographical order $<$ on the sets from ${[n]\choose k}$ by setting $A<B$ iff either $B \subset A$ or the minimal element of $A\setminus B$ is less than the minimal element of $B\setminus A$. We  note that, in particular, $A<A$.

 For $0\le m\le {n\choose k}$ let $\mathcal L(m,k)$ be the collection of $m$ largest $k$-sets with respect to this order. In the proof we are going to use the famous Kruskal-Katona Theorem \cite{Ka}, \cite{Kr} in the form due to Hilton \cite{Hil}:

\begin{thm}\label{thmHil}Suppose that $\mathcal A\subset {[n]\choose a}, \mathcal B\subset {[n]\choose b}$ are cross-intersecting. Then the same holds for the families $\mathcal L(|\mathcal A|,a),\mathcal L(|\mathcal B|, b).$
\end{thm}

We may demonstrate the power of Theorem \ref{thmHil} immediately, proving the following corollary of it and Theorem \ref{thmcross}:

\begin{cor}\label{corcross} The statement of Theorem \ref{thmcross} holds even if we replace the condition $\mathcal B\subset \mathcal A$ by $\mathcal B\subset {[n]\choose k}, |\mathcal B|\le |\mathcal A|$.
\end{cor}
\begin{proof} The proof is straightforward. One has to pass to the families $\mathcal L(|\mathcal A|,k),\mathcal L(|\mathcal B|, k)$. Then the condition $|\mathcal B|\le |\mathcal A|$ is equivalent to $\mathcal L(|\mathcal B|,k)\subset \mathcal L(|\mathcal A|, k)$. After we just have to apply Theorem \ref{thmcross} to the families $\mathcal L(|\mathcal A|,k),\mathcal L(|\mathcal B|, k)$.\end{proof}

To avoid trivialities, for the whole section we assume that $a+b\le n$.

We say, that two sets $S$ and $T$ \textit{strongly intersect}, if they intersect and for the first $j\in S\cap T$ we have $[j]\subset S\cup T$.
Let $S$ be a finite $s$-element set and $t\ge s$ an integer. Define the set family $\mathcal L(S,t)$ by
 $$\mathcal L(S,t) = \{ T \in{[n] \choose t} : T<S\}$$


\begin{prop}\label{cross1} Let $A$ and $B$ be an $a$-element and a $b$-element set, respectively. Assume that $n\ge a+b$. Then $\mathcal L(A,a)$ and $\mathcal L(B,b)$ are cross-intersecting iff $A$ and $B$  strongly intersect.\end{prop}

\begin{proof} First suppose that $\mathcal L(A,a)$ and $\mathcal L(B,b)$ are cross-intersecting. Then $A$ and $B$ intersect. Let $j$ be the smallest integer contained in both. If $A\cup B$ contains $[j]$, then $A$ and $B$ strongly intersect. Otherwise, there exists $i<j$ such that $i\notin A\cup B$. Since $n\ge a+b$, there exists an $a$-element set $C$ satisfying $C\cap [i] = A\cap [i]\cup \{i\}$, which is \textit{disjoint} with $B$. At the same time, $C<A$, contradicting the assumption that $\mathcal L(A,a)$ and $\mathcal L(B,b)$ are cross-intersecting.

Now suppose that $A$ and $B$ strongly intersect. Let $C<A$. We claim that $C$ and $B$ strongly intersect. If $C\cap [j] = A\cap [j],$ then it follows directly from the definition. Otherwise, let $i$ be the first element where they differ. Since $C<A$, we have $i\in C\backslash A,\ i<j.$ But then $i\in C\cap B,\ [i]\subset C\cup B$. Repeating the same argument for any $D$ with $D<B$ gives that $C$ and $D$ strongly intersect. In particular, they have a non-empty intersection. Therefore, $\mathcal L(A,a)$ and $\mathcal L(B,b)$ are cross-intersecting.
\end{proof}

\begin{defn} We say that $\mathcal A\subset {[n]\choose a}$ and $\mathcal B\subset {[n]\choose b}$ form a \textit{maximal cross-intersecting pair}, if whenever $\mathcal A'\subset {[n]\choose a}$ and $\mathcal B'\subset {[n]\choose b}$ are cross-intersecting with $\mathcal A'\supset \mathcal A$ and $\mathcal B'\supset \mathcal B$, then necessarily $\mathcal A = \mathcal A'$ and $\mathcal B = \mathcal B'$ holds.
\end{defn}

Now we are in a position to prove the following strengthening of Theorem \ref{thmHil}. We believe that this proposition is of independent interest. It was definitely of great use in proving Theorem \ref{thmsmall}.

\begin{prop}\label{cross2} Let $a$ and $b$ be positive integers, $a+b\le n$. Let $P$ and $Q$ be non-empty subsets of $[n]$ with $|P|\le a$, $|Q| \le b$. Suppose that $P$ and $Q$ strongly intersect in their last element. That is, there exists $j$, such that $P\cap Q = \{j\}$ and $P\cup Q = [j]$. Then $\mathcal L(P,a)$ and $\mathcal L(Q,b)$ form a maximal pair of cross-intersecting families.

Inversely, if $\mathcal L(m,a)$ and $\mathcal L(r,b)$ form a maximal pair of cross-intersecting families, then it is possible to find sets $P$ and $Q$ such that $\mathcal L(m,a)=\mathcal L(P,a)$, $\mathcal L(r,b)=\mathcal L(Q,b)$ and $P,Q$ satisfy the above condition. \end{prop}

 \begin{proof} If $P$ and $Q$ satisfy the condition then $\mathcal L(P,a)$ and $\mathcal L(Q,b)$ form a  pair of cross-intersecting families by Proposition \ref{cross1}. We have to show that it is a maximal pair. Let $A$ be the $a$-set, following $\mathcal L(P,a)$ in the lexicographical order. Then $A$ is the set obtained from $P$ by omitting the element $\{j\}$ and adjoining the interval $[j+1,p]$, where $p=a-|P|+1$. Note that $A\cap Q =\emptyset$ , therefore there are members of $\mathcal L(Q,b)$ containing $Q$, which are disjoint with $A$. The same argument applies to the set $B$, following the initial segment $\mathcal L(Q,b)$. Thus, the maximality of the pair $\mathcal L(P,a)$ and $\mathcal L(Q,b)$ is proved.\\

 Next, let $A$, $B$ be the last member of $\mathcal L(m,a)$ and $\mathcal L(r,b)$, respectively. Let $j$ be the smallest element of $A\cap B$. Note that the cross-intersecting property of the two families implies that $j$ is well-defined. Define $P=A \cap [j], Q= B \cap [j].$

 First we prove that $P \cup Q =[j]$. Suppose the contrary and let $i<j$ be an element that is not contained in $P \cup Q$. Then $P'=(P-\{j\})\cup {i}$ precedes $P$ in the lexicographical order and hence also $P'<A$. Consequently, all $a$-element supersets of $P'$ precede $A$ as well. Thus they are all members of $\mathcal L(m,a)$. However, $P'\cap B = \emptyset$  and, since $a+b \le n$ we have that some superset $P'$ is disjoint to $B$, a contradiction.

 The proof is almost complete. We have just proved that $P$ and $Q$ are strongly intersecting. By Proposition \ref{cross1} the families $\mathcal L(P,a)$ and $\mathcal L(Q,b)$ form a  pair of cross-intersecting families. These families contain $\mathcal L(m,a)$ and $\mathcal L(r,b)$, respectively. By the maximality of the pair $\mathcal L(m,a)$ and $\mathcal L(r,b)$, $\mathcal L(m,a)=\mathcal L(P,a),\ \mathcal L(r,b)=\mathcal L(Q,b)$ must hold.  \end{proof}

In what follows we will use the following simple statement:
\begin{prop}\label{prop11} Let $\mathcal A\subset{[n]\choose a}$ and $\mathcal B\subset {[n]\choose b}$ be cross-intersecting. Then we have $|\mathcal A|+c|\mathcal B|\le \max\left\{{n\choose a},c{n\choose b}\right\}.$
\end{prop}
\begin{proof} Consider the bipartite graph with parts ${[n]\choose a}$ and  ${[n]\choose b}$, in which two sets connected by an edge iff they are disjoint. Assign weight 1 to each vertex in ${n\choose a}$ and weight $c$ to any vertex in ${n\choose b}$. Then   the proposition essentially states that the independent set of the biggest weight in this graph coincides with one of its two parts. It is an easy consequence of the fact that the graph is regular in each of its parts. Denote the fraction of vertices from the independent set in the part ${[n]\choose a}$ and ${[n]\choose b}$ by $x$ and $y$, respectively. To see that, take an edge in the graph at random. On the one hand, at most one vertex of the independent set may be in the edge. On the other hand, the expected intersection of the edge with the independent set is $x+y$ (due to regularity on both sides). Therefore, $x+y\le 1$ and, clearly, ${n\choose a}x+c{n\choose b}y\le \max\{{n\choose a}, c{n\choose b}\}$.    \end{proof}

We are interested in bounding $|\mathcal L(m,a)|+c|\mathcal L(r,b)|$ from above, and thus we may w.l.o.g. restrict our attention only to maximal intersecting pairs of families, which, by Proposition \ref{cross2}, are equal to $\mathcal L(P,a), \mathcal L(Q,b)$ for some $P,Q\subset [n]$ with $P\cup Q=[i],\ P\cap Q=\{i\}$ for some $i\in [n]$. We will call such $P, Q$ {\it defining sets} for $\mathcal L(P,a),\ \mathcal L(Q,b)$, respectively (or that $\mathcal L(P,a)$ is {\it defined by $P$}). All the families will be lexicographical.


These easy-to-prove lemmas actually give a very useful consequence concerning the sizes of maximal cross-intersecting pairs of families $\mathcal L(m,a),\mathcal L(r,b)$.

\begin{lem}\label{lem12} Consider two maximal cross-intersecting families $\mathcal A = \mathcal L(P,a),\ \mathcal B = \mathcal L(Q,b)$, such that $P\cap Q=\{i\}$, $P\cup Q=[i]$. Take any $j<i$ such that $j\in Q$, $j+1\notin Q$ (and thus $j\notin P,\ j+1\in P$).


Consider the following two pairs of cross-intersecting lexicographical families. The first pair is $\mathcal A',\mathcal B'$, defined by $P'=(P\cap[j])\cup \{j\}$,  $Q' = Q\cap [j]$. The second pair is $\mathcal A'',\mathcal B''$, defined by $P''=P\cap[j+1]$,  $Q'' = (Q\cap [j+1])\cup \{j+1\}$.


Then for any $c>0$ we have
\begin{equation}\label{eq11} |\mathcal A|+c|\mathcal B|\le \max \{|\mathcal A'|+c|\mathcal B'|, |\mathcal A''|+c|\mathcal B''|\}.\end{equation}
\end{lem}
For integers $i<j$ we denote $[i,j]=\{i,i+1,\ldots,j\}$. Before going into the proof of the lemma, we introduce the following notation. For $i\in [n]$, a set $S\subset [n]$ and a family $\mathcal F\subset {[n]\choose k}$ we define $$\mathcal F_i(S) = \{F'\in {[i+1,n]\choose k-|S\cap[i]|}: F'\cup (S\cap[i])\in \mathcal F\}.$$
Let us also denote  $\mathcal Z={[n]\choose k}.$


\begin{proof} 
 We remark that, by definition, we have $\mathcal B''\subset \mathcal B\subset \mathcal B'$ and $\mathcal A''\supset \mathcal A\supset \mathcal A'.$
Moreover, we have $\mathcal B' - \mathcal B'' = \mathcal Z_{j+1}(Q)$. Similarly, we have $\mathcal A'' - \mathcal A'=\mathcal Z_{j+1}(P)$.
We also have $\mathcal A-\mathcal A'=\mathcal A_{j+1}(P)$ and $\mathcal B-\mathcal B''=\mathcal B_{j+1}(Q)$.

Since $P\cap Q\cap[j+1]=\emptyset$, we know that $\mathcal A_{j+1}(P)$ and $\mathcal B_{j+1}(Q)$ are cross-intersecting. Therefore, by Proposition \ref{prop11} we have
\begin{equation}\label{eq12}|\mathcal A_{j+1}(P)| +c|\mathcal B_{j+1}(Q)|\le \max\{|\mathcal Z_{j+1}(P)|, c|\mathcal Z_{j+1}(Q)| \} = \max\{|\mathcal A'' - \mathcal A'|, c|\mathcal B' - \mathcal B''| \}.\end{equation}
\end{proof}



Using almost identical proof, one may get the following  twin of Lemma \ref{lem12}:

\begin{lem}\label{lem13}Consider two maximal cross-intersecting families $\mathcal A = \mathcal L(P,a),\ \mathcal B = \mathcal L(Q,b)$, such that $P\cap Q=\{i\}$, $P\cup Q=[i]$. Moreover, assume that for some $j< i$ $[j]\subset Q$.

Consider the pair of cross-intersecting lexicographical families $\mathcal A',\mathcal B'$, defined by $P'=\{j\}$,  $Q' = [j]$.




Then for any $c>o$ we have
\begin{equation}\label{eq13} |\mathcal A|+c|\mathcal B|\le \max \left \{|\mathcal A'|+c|\mathcal B'|, {n\choose a}\right\}.\end{equation}
\end{lem}
\begin{proof} Note that $\mathcal B\subset \mathcal Z_j(Q)=\mathcal B'$. Similarly, $\mathcal A_j(P)=\mathcal A-\mathcal A'\subset \mathcal Z_j(P)$. The families $\mathcal A_j(P)$ and $\mathcal B_j(Q)=\mathcal B$ are cross-intersecting. Therefore,
$$|\mathcal A_j(P)|+c|\mathcal B|\le \max\{|\mathcal Z_j(P)|, c|\mathcal Z_j(Q)|\}.$$
Therefore,
$$|\mathcal A|+c|\mathcal B|\le \max \big\{|\mathcal A'|+|\mathcal Z_j(P)|,|\mathcal A'|+c|\mathcal Z_j(Q)|\big\}\le \max \left \{{n\choose a}, |\mathcal A'|+c|\mathcal B'| \right\}.$$
\end{proof}


\subsection{A sharpening of Theorem \ref{thmcross}}\label{sec46}
In the proof of Theorem \ref{thmrec2} we need a sharpened version of Theorem \ref{thmcross}, which proof relies on the material from the previous subsection.
For integer $j\ge 2$ and $c>0$ let us denote
\begin{equation}\label{eqf} f_{n,k}(j,c)= (c+1){n-1\choose k-1}-c{n-j\choose k-1} +{n-j\choose k-j+1}.\end{equation}
Note that $f_{n,k}(j,c) = |\mathcal L(P,k)|+c|\mathcal L(Q,k)|$, where $P=[2,j]$ and $Q=\{1,j\}$.

\begin{thm}\label{thmcross2} Let $k, n\in \N$,where $ k\ge 2$ and $n\ge 2k$.
Consider two cross-intersecting families $\mathcal A,\ \mathcal B$, where  $\mathcal B\subset \mathcal A\subset {[n]\choose k}$.  Let $2\le i\le k+1$ and $|\mathcal B|\le {n-1\choose k-1}-{n-i\choose k-1}$. Then for any $c\ge 1$
\begin{equation}\label{eqcross2} |\mathcal A|+c|\mathcal B|\le \max\Bigl\{f_{n,k}(2,c),f_{n,k}(i,c), {n\choose k}\Bigr\}.
\end{equation}
\end{thm}
\begin{proof} If $\mathcal B$ is empty, then the statement is obvious, therefore, we assume the opposite for the rest of the proof. By Theorem \ref{thmHil}, $\mathcal L_0 = \mathcal L(|\mathcal A|, k)$ and $\mathcal L_1 = \mathcal L(|\mathcal B|, k)$ are cross-intersecting. W.l.o.g. we may assume that this is a maximal cross-intersecting pair, and, thus, by Proposition \ref{cross2}, $\mathcal L_0 = \mathcal L(P,k)$ and $\mathcal L_1 = \mathcal L(Q,k)$, where $P\cap Q = \{l\}$ and $P\cup Q = [l]$  for some $l\ge 2$. Moreover, due to the restrictions on $|\mathcal B|$, we know that $1\in P$, and $|P\cap [i]|\ge 2$. Let $i_0$ be the smallest integer, for which $|P\cap [i_0]|\ge 2$. Then, obviously, $|P\cap [i_0]|=2$ and $i_0\le i$.



\textbf{(i)\ } Assume that $i_0=2$. Then, applying Lemma~\ref{lem13} to $\mathcal A = \mathcal L_0, \ \mathcal B=\mathcal L_1$ with $j=2$ and $a=b=k$, we get that $|\mathcal L_0|+c|\mathcal L_1|\le \max\big\{f_{n,k}(2,c), {n\choose k}\big\}$.

\textbf{(ii)\ } Assume that $i_0>2$. If $Q = \{1,i_0\}$, then $|\mathcal L_0|+c|\mathcal L_1| = f_{n,k}(i_0,c)$. Assume further that $\{1,i_0\}\subsetneq Q$.
Then, in particular, $i_0-1\in P$ and $i_0\notin P$.
Therefore, we may apply Lemma~\ref{lem12} with $\mathcal A = \mathcal L_1,\ \mathcal B = \mathcal L_0$, $j=i_0-1$, $a=b=k$ and get that
$$c(|\mathcal L_1|+\frac 1c|\mathcal L_0|)\le c\max\Big\{|\mathcal L(Q',k)|+\frac 1c|\mathcal L(P',k)|, |\mathcal L(Q'',k)|+\frac 1c|\mathcal L(P'',k)\Big\},$$
where $P'=[2,i_0-1],\ Q'=\{1,i_0-1\}$ and $P''=[2,i_0],\ Q''=\{1,i_0\}$. Thus, the right hand side of the displayed inequality above is precisely
$$\max\{f_{n,k}(i_0-1,c), f_{n,k}(i_0,c)\}.$$
To complete the proof, we need to show that for $j\ge 3$ we have $f_k(j,c)\le f_k(j+1,c)$.
Indeed,
\begin{multline}\label{eq632} f_{n,k}(j+1,c)-f_{n,k}(j,c) = c\Bigl({n-j\choose k-1}-{n-j-1\choose k-1}\Bigr) \\- \Bigl({n-j\choose k-j+1}-{n-j-1\choose k-j}\Bigr)= c{n-j-1\choose k-2}- {n-j-1\choose k-j+1}\ge 0,\end{multline}
where the last inequality holds due to the fact that $c\ge 1$ and
$$\frac{{n-j-1\choose k-2}}{ {n-j-1\choose k-j+1}} = \prod_{l = 2}^{j-2}\frac{n-k-l}{k-l}\ge 1,$$
since $n\ge 2k$. This completes the proof of the theorem.
\end{proof}

For our purposes it would be more convenient to apply the following slightly modified version of the theorem:
\begin{cor}\label{thmcross3} In the conditions of Theorem \ref{thmcross2} let $2\le i\le k$ and $|\mathcal B| \le  {n-1\choose k-1}-{n-i\choose k-1}+x$ for some natural $x$. Then
\begin{equation}\label{eqcross3} |\mathcal A|+c|\mathcal B|\le \max\Bigl\{f_{n,k}(i,c)+cx, f_{n,k}(2,c), {n\choose k}\Bigr\}.
\end{equation}
\end{cor}
\begin{proof} The proof is basically the same, the only thing one has to notice that, if $|\mathcal B|$ gets bigger, then $|\mathcal A|$ can only get smaller. Therefore, if $|\mathcal B| =  {n-1\choose k-1}-{n-i-1\choose k-1}+y$, $1\le y\le x$, then one just remove
$y$ elements from $|\mathcal B|$, applies the same proof, and puts the elements back, adding $cy$ to the right hand side. Note that we only add $cx$ to the first expression, since the other two appear in the proof only when $|\mathcal B|$ is relatively small.
\end{proof}


\subsection{Intersecting families with conditions on maximum degree}\label{sec47}

For a given family $\mathcal F$ let $d(\mathcal F) = \max_{l\in [n]}|\{F\in \mathcal F: l\in F\}|$ be the maximal degree of an element of $[n]$ in $\mathcal F$. The diversity $\gamma(\mathcal F)$ of $\mathcal F$ is defined as $|\mathcal F|-d(\mathcal F)$. The other important ingredient of the proof of Theorem \ref{thmsmall} is the following theorem due to Frankl \cite{Fra1} in the version due to Kupavskii and Zakharov \cite{KZ}. (See also \cite{Kup17} for a yet stronger version of the theorem)
\begin{thm}\label{frdeg} Suppose that $n>2k$, $3\le i\le k+1$, $\mathcal F\subset{[n]\choose k}$, $\mathcal F$ is intersecting and
$$\gamma (\mathcal F)\ge  {n-i\choose k-i+1}.$$
Then
\begin{equation*}\label{eqsize}|\mathcal F|\le {n-1\choose k-1}-{n-i\choose k-1} +{n-i\choose k-i+1} =: g(i).\end{equation*}
In particular, if $\gamma(\mathcal F)>0$, then $|\mathcal F|\le g(k+1)$ holds.
\end{thm}
The theorem is sharp for each $i$, as suggested by the following example:
$$\Big\{F\in{[n]\choose k}: 1\in F, [2,i]\cap F\ne\emptyset\Big\}\cup\Big\{F\in {[n]\choose k}: [2,i]\subset F\Big\}.$$

It is not difficult to see that $g(3) =g(4)$ and $g(i)<g(j)$ if  $4\le i<j$. Let us verify the first claim:\begin{small}
$$g(3)={n-1\choose k-1}-{n-3\choose k-1}+{n-3\choose k-2} = {n-1\choose k-1}-{n-4\choose k-1}-{n-4\choose k-2}+{n-4\choose k-2}+{n-4\choose k-3}=g(4).$$
\end{small}

\section{Proof of Theorem \ref{thmsmall} in the case  $2k\le n\le 3k-1$}\label{sec4}
In this section we show that  $\mathcal E(n,k,1)$ has the maximum cardinality among intersecting families for $2k\le n\le 3k-1$.

 The proof is based on the application of the general Katona's circular method directly for $\mathcal V(n,k,1)$. Consider the following subfamily $\mathcal H$ of $\mathcal V(n,k,1)$:
$$\mathcal H = \{\mathbf v = (v_1,\ldots, v_n): \text{ for some } i\in[n]\ v_i = \ldots = v_{i+k-1} = 1,\ v_{i-k} = -1\}.$$
We remark that all indices are modulo $n$. That is, it is a usual Katona's circle for $k$-sets, but in which each $k$-set gets an extra $-1$-coordinate, which is at distance $k$ from the 1-part along the circle.

Take an intersecting family $\mathcal F\subset \mathcal V(n,k,1)$. We claim that $|\mathcal F\cap \mathcal H|\le k$. Denote by $\mathcal F'\subset{[n]\choose k}$ the family of sets of 1's from $\mathcal F$, and similarly for $\mathcal H'$. We claim that $\mathcal H'\cap \mathcal F'$ is an intersecting family. Assume that there are two sets $F'_1,F'_2\in \mathcal H'\cap \mathcal F'$, that are disjoint. Assume for simplicity  that $F'_2 = [k+1,2k]$. Then $F'_1$ is obliged to contain $\{1\}$, since any cyclic interval of length $k$ in $[n]\setminus [k+1,2k]$ contains $\{1\}$, provided that $n\le 3k-1$. Therefore, the corresponding vector $F_1\in \mathcal F\cap \mathcal H$ has 1 on the first coordinate position. At the same time, by definition of $\mathcal H$, the vector $F_2$ has -1 on the first coordinate position. Interchanging the roles of $F_1,F_2$, we get that both of them have -1 in front of 1 of the other vector. Moreover, their sets of 1's do not intersect. Therefore, they form a minimal scalar product. Therefore, $\mathcal H'\cap \mathcal F'$ is indeed intersecting, and $|\mathcal H\cap \mathcal F| = |\mathcal H'\cap \mathcal F'| \le k$.

The rest of the argument is an application of Lemma \ref{lem66} to the following graph. Let $G$ have the set of vertices $\mathcal V(n,k,1)$, with two vertices connected if the corresponding vectors have scalar product -2. Then $|\mathcal F|\le \alpha(G)$ and the considerations from the previous paragraph give that $\alpha(G|_{\mathcal H}) = k$. By Lemma \ref{lem66} we have $$\alpha(G)\le \frac{\alpha(G|_{\mathcal H})}{|\mathcal H|}|\mathcal V(n,k,1)|= \frac kn|\mathcal V(n,k,1)| = e(n,k,1).$$
The last equality was obtained in Section \ref{sec31}.

\section{Proof of Theorem \ref{thmbig}}\label{sec5}
For any intersecting family $\mathcal G\subset\mathcal V(n+1,k,1)$ we introduce the following notations. By $\mathcal G(i),\mathcal G(-i),\mathcal G(\bar i)$ we denote the subfamilies of $\mathcal G$, that have $1,-1,0$ as an $i$-th coordinate, respectively. It is important to mention that we consider them as families of vectors on the set of coordinates with the $i$-th coordinate excluded.
The definition extends in an obvious way on the family $\mathcal G(I_1, \bar I_2,-I_3)$, where $I_1,I_2,I_3\subset[n]$ are non-intersecting sets of indices.


Consider a maximum intersecting family $\mathcal F\subset\mathcal V(n+1,k,1)$. Based on the conclusion of Section \ref{sec321}, we may and will assume that $\mathcal F$ is shifted. Denote $\mathcal A = \mathcal F(-(n+1))$ and $\mathcal B = \mathcal F(n+1)$.

\begin{prop}\label{prop3} We have $m(n+1,k,1)-m(n,k,1)\le |\mathcal A|+|\mathcal B|.$
\end{prop}
\begin{proof} First, by definition we have $|\mathcal F| = m(n+1,k,1).$ Second, consider a family $\mathcal F(\overline{n+1})$.
It is a subfamily in $\mathcal V(n,k,1)$, moreover, it is intersecting. Therefore, $|\mathcal F(\overline{n+1})|\le m(n,k,1)$. Finally, $m(n+1,k,1) =|\mathcal F| = |\mathcal F(\overline{n+1})|+|\mathcal A|+|\mathcal B|$. \end{proof}

For a given $i\in [n]$ consider two families of sets $\mathcal B(-i),\mathcal A(i)$ (both of these families could be seen as sets, since the only $-1$-coordinate is fixed and excluded in both). We have $\mathcal B(-i),\mathcal A(i)\subset{[n]- \{i\}\choose k-1}.$ Moreover, $\mathcal B(-i)\subset\mathcal A(i)$ due to shifting, and we may apply Theorem \ref{thmcross} to these two families with $c = k$ and obtain
$$|\mathcal A(i)|+k|\mathcal B(-i)|\le \max\Bigl\{{n-1\choose k-1}, (k+1){n-2\choose k-2}\Bigr\}.$$
Summing this inequality over all $i\in [n]$, we get
\begin{equation*}\sum_{i=1}^n|\mathcal A(i)|+k|\mathcal B(-i)| = k(|\mathcal A|+|\mathcal B|)\le n \max\Bigl\{{n-1\choose k-1}, (k+1){n-2\choose k-2}\Bigr\}\end{equation*}
\begin{equation}\label{eq5} |\mathcal A|+|\mathcal B| \le \max\Bigl\{{n\choose k}, \frac{n(k+1)}k{n-2\choose k-2}\Bigr\}.\end{equation}
Maximum in the right hand side of (\ref{eq5}) is attained on the first expression if

$${n\choose k} = \frac {n(n-1)}{k(k-1)}{n-2\choose k-2}\ge \frac {n(k+1)}k{n-2\choose k-2},$$
which is equivalent to $n\ge k^2$.

\begin{proof}[Proof of Theorem \ref{thmbig}] The bound $m(n+1,k,1) \le m(n,k,1)+{n\choose k}$ was, in fact, already proven in this section. In the notations above, consider $\mathcal F, \mathcal A, \mathcal B$. On the one hand, by (\ref{eq5}) and the discussion after this inequality,  we have $|\mathcal A|+|\mathcal B|\le {n\choose k}$ for $n\ge k^2$. Applying Proposition \ref{prop3}, we get the bound.

The bound  $m(n+1,k,1) \ge m(n,k,1)+{n\choose k}$ was already obtained in Section \ref{sec31} (and mentioned in the introduction).
\end{proof}


\section{Proof of Theorem \ref{thmsmall} in the case $3k\le n\le k^2$}\label{sec6}

Looking at equation (\ref{eq5}) in the case $n<k^2$, we see that $m(n+1,k,1)-m(n,k,1)\le  \frac{n(k+1)}k{n-2\choose k-2}$. On the other hand, from Section \ref{sec31} we know, that $e(n+1,k)-e(n,k) = k{n-1\choose k-1} = \frac{(n-1)k}{k-1}{n-2\choose k-2}$. We have $\frac{n(k+1)}k-\frac{(n-1)k}{k-1} = \frac{k^2-n}{k(k-1)}$. Using Theorem \ref{thmcross2}, we are going to improve the inequality (\ref{eq5}), so that it matches the bound given by the construction $\mathcal E(n,k,1)$.
To complete the proof of Theorem \ref{thmsmall}, it is enough to prove the following theorem.
\begin{thm}\label{thmrec2}
Let $k\ge 3$, $3k-1\le n<k^2$. We have  $m(n+1,k,1)-m(n,k,1) = \frac{(n-1)k}{k-1}{n-2\choose k-2} = e(n+1,k)-e(n,k)$.
\end{thm}
Note that for $k=2$ we have $3k-1>k^2$, thus, there is nothing to prove in this case.
Before proving the theorem, we state and prove the following proposition:

\begin{prop}\label{prop4} Let $n\ge 3k-1$. In terms of Section~\ref{sec5}, consider a maximum intersecting family $\mathcal F\subset\mathcal V(n+1,k,1)$ and families of vectors $\mathcal A,$ $\mathcal B$ in $\{0,\pm 1\}^n$. There is a subset $I\subset[n]$, $|I|\ge \lceil \frac {3k}2\rceil$, such that for every $l\in I$ we have $|\mathcal B(\bar 1,-l)|\ge  \frac 13 |\mathcal B(-1)|.$
\end{prop}
\begin{proof} Take any $B\in \mathcal B(-1)$.  Then, due to the fact that $\mathcal F$ is shifted, if we swap -1, which is on the first coordinate position, and some of the 0's in $B$, we obtain a set $B'\in\mathcal B(\bar 1).$ Since there are $n-k$ zeros in each $B\in \mathcal B(-1)$, we may obtain $n-k$ sets $B'\in \mathcal B(\bar 1)$ out of each $B$. Moreover, any two sets obtained are different. Indeed, for  $B$ any two vectors obtained out of it have different positions of -1, while any two vectors obtained from different $B_1,B_2\in \mathcal B(-1)$  have different sets of 1's. Thus, $|\mathcal B(\bar 1,-2)|+\ldots+|\mathcal B(\bar 1,-n)|\ge (n-k)|\mathcal B(-1)|$. By pigeon-hole principle we get that one of the summands must be at least $\frac {n-k}{n-1} |\mathcal B(-1)|$. This is the first element from~$I$.

Once we have found the $i$-th element, which satisfies the inequality from the proposition, we add it to $I$, and delete this element from all the vectors from $\mathcal B(-1)$ and from the ground set. The set of already found elements we denote $I_{i}$. After $i$ steps each of the sets from the modified $\mathcal B(-1)$ has at least $n-k-i$ zeros, while the total number of coordinates is $n-i-1$. Therefore, the inequality from the previous paragraph, modified for this case, looks like $\sum_{j\in[2,n]\setminus I_{i}}|\mathcal B(\bar 1,-j)|\ge (n-k-i)|\mathcal B(-1)|,$ and by pigeon-hole principle we can find an element $l_i$ such that $|\mathcal B(\bar 1,-l_i)|\le \frac {n-k-i}{n-i-1} |\mathcal B(-1)|,$ which is bigger than $\frac 13|\mathcal B(-1)|$ for $i+1\le \lceil 3k/2\rceil$ and $n\ge 3k-1$.
\end{proof}
\begin{proof}[Proof of Theorem \ref{thmrec2}] We argue in terms used in Section \ref{sec5}. We consider several cases depending on the size of $\mathcal B(-1)$.\\

\textbf{(i)\ } First, assume that $\mathcal B(-1)$ is empty. Then, applying Theorem \ref{thmcross} with $c = k$, we get $|\mathcal A(i)|+k|\mathcal B(-i)|\le (k+1){n-2\choose k-2}$ for $i = 2,\ldots, n$ (see \eqref{eq5} and the calculations after it) and $|\mathcal A(1)|+k|\mathcal B(-1)| \le {n-1\choose k-1}$. Therefore,

\begin{equation*}\sum_{i=1}^n|\mathcal A(i)|+k|\mathcal B(-i)| = k(|\mathcal A|+|\mathcal B|)\le {n-1\choose k-1} + (n-1)(k+1){n-2\choose k-2}\end{equation*}
\begin{equation*}|\mathcal A|+|\mathcal B|\le \Bigl(\frac{n-1}{(k-1)k} + \frac{(n-1)(k+1)}k\Bigr){n-2\choose k-2} = \frac{(n-1)k}{k-1}{n-2\choose k-2} = e(n+1,k)-e(n,k). \end{equation*}

It is apparent from the calculations above that the theorem follows if we succeed to show that for some $I\subset[2,n]$ we have $\sum_{l\in I}\bigr(|\mathcal A(l)|+k|\mathcal B(-l)|\bigl)+k|\mathcal B(-1)|\le |I|(k+1){n-2\choose k-2}$. This is exactly what are we going to do in a range of cases. \\

\textbf{(ii)\ } Assume that $|\mathcal B(-1)|\ge 3{n-4\choose k-3}.$ Then, by Proposition \ref{prop4}, we can find $I\subset [2,n]$, $|I|\ge \lceil3k/2\rceil\ge 5$, such that we have $|\mathcal B(\bar 1, -l)|\ge  {n-4\choose k-3}$.
For any $l\in I$ consider the collection $\mathcal B(-l)$. Due to shifting, we know that the maximum degree $d(\mathcal B(-l))$ is equal to the number of sets from $\mathcal B(-l)$ that have 1 on the first coordinate position. Therefore, we know that $\gamma(\mathcal B(-l))=|\mathcal B(\bar 1, -l)| \ge {n-4\choose k-3}$ and, thus, we may apply Theorem~\ref{frdeg} to $\mathcal B(-l)$ with $i = 3$ and $n,k$ replaced by $n-1,k-1$. We obtain that

$$|\mathcal B(-l)|\le {n-2\choose k-2}-{n-4\choose k-2} +{n-4\choose k-3}.$$

For each $l\in I$ we apply Corollary~\ref{thmcross3} to $\mathcal B(-l),\mathcal A(-l)$ with $c = k$, $x={n-4\choose k-3}$ and $i=3$ and obtain that
\begin{equation*} |\mathcal A(l)|+k|\mathcal B(-l)|\le \max\Bigl\{f_{n-1,k-1}(3,k)+k{n-4\choose k-3},f_{n-1,k-1}(2,k), {n-1\choose k-1}\Bigr\}.
\end{equation*}
If  the maximum of the right hand side is attained on the third summand, then for a single $l\in I$
$$|\mathcal A(l)|+k|\mathcal B(-l)|+|\mathcal A(1)|+k|\mathcal B(-1)|\le {n-1\choose k-1}+(k+1){n-2\choose k-2}$$
and we are done. Recall that, by \eqref{eqf}, $f_{n-1,k-1}(2,k)=(k+1){n-2\choose k-2}-(k-1){n-3\choose k-2}$. Note also that $|\mathcal B(-1)|\le {n-2\choose k-2}$ since $\mathcal B(-1)$ is intersecting. If the maximum is attained on the second expression, then, since $|I|\ge 5$, for five different $l_1,\ldots,l_5$  we have
$$\sum_{j=1}^5(|\mathcal A(l_j)|+k|\mathcal B(-l_j)|)+k|\mathcal B(-1)|- (5k+5){n-2\choose k-2} \le k{n-2\choose k-2}-5(k-1){n-3\choose k-2}\le$$$$\le (2k-5(k-1)){n-3\choose k-2}< 0,$$
since $k\ge 3$, and we are done. Recall that, by \eqref{eqf}, $f_{n-1,k-1}(3,k)=(k+1){n-2\choose k-2}-k{n-4\choose k-2}+{n-4\choose k-3}$. Finally, if the maximum  is attained on the first summand, then
\begin{align}\frac 1{|I|}&\Bigl(\sum_{l\in I} (|\mathcal A(l)|+ k|\mathcal B(-l)|)+k|\mathcal B(-1)|\Bigr)-(k+1){n-2\choose k-2}\le\notag \\\label{eq9} \le&-k{n-4\choose k-2}+ (k+1){n-4\choose k-3}+\frac k{|I|}{n-2\choose k-2}.\end{align}
Before continuing, we need the following two bounds, valid for $n\ge 3k-1$:

 \begin{equation*}\frac{(k+1){n-4\choose k-3}}{k{n-4\choose k-2}} = \frac{(k+1)(k-2)}{k(n-k-1)}\le \frac 12. \end{equation*}
 $${n-2\choose k-2} = \frac{(n-2)(n-3)}{(n-k)(n-k-1)}{n-4\choose k-2}\le \frac {9}4{n-4\choose k-2}.$$
Using these three bounds and $|I|\ge $, we get
\begin{align*}(\ref{eq9})\le &-\frac 12 k{n-4\choose k-2}+ \frac {9}4\cdot\frac k{5}{n-4\choose k-2}< 0.
\end{align*}
This case is settled.\\

\textbf{(iii)\ } Assume that for some $4\le i \le k$ we have $$3{n-i-1\choose k-i}\le|\mathcal B(-1)|\le 3{n-i\choose k-i+1}.$$
Similarly to the previous case, for $I\subset [n], |I| \ge 5$, we obtain
$$|\mathcal B(-l)|\le {n-2\choose k-2} - {n-i-1\choose k-2}+{n-i-1\choose k-i}.$$
Note that if $0<|\mathcal B(-1)|<3{n-k-1\choose k-k}=3$ we have $\mathcal B(\bar 1,-l)\ge \frac 13|\mathcal B(-1)|>0$ and we get the same bound on $|\mathcal B(-l)|$ as the one above for $i=k$, so this case is also covered and cases (i), (ii), (iii) altogether cover all possible values of $|\mathcal B(-1)|$.

 We again apply Corollary \ref{thmcross3} for each $\mathcal B(-l),\mathcal A(-l)$ with $x = {n-i-1\choose k-i}$. If the maximum of the expression  from the inequality (\ref{eqcross3}) is attained on one of the last two expressions, we are done, since we can do exactly the same as in case (ii). Recall that $f_{n-1,k-1}(i,k)=(k+1){n-2\choose k-2}-k{n-i-1\choose k-2}+{n-i-1\choose k-i}$. If it is attained on the first one, then

\begin{align}\frac 1{|I|}&\Bigl(\sum_{l\in I} (|\mathcal A(l)|+ k|\mathcal B(-l)|)+k|\mathcal B(-1)|\Bigr)-(k+1){n-2\choose k-2}\le\notag \\\le &-k{n-i-1\choose k-2}+(k+1){n-i-1\choose k-i}+\frac k{|I|}|\mathcal B(-1)|\le \notag \\\le \label{eq10}
&-k{n-i-1\choose k-2}+(k+1){n-i-1\choose k-i}+\frac {3k}{5}{n-i\choose k-i+1}.\end{align}
To proceed further, we need the following bounds, valid for $n\ge 3k-1$:

\begin{equation*}\frac{(k+1){n-i-1\choose k-i}}{k{n-i-1\choose k-2}} = \frac{k+1}k\prod_{j=2}^{i-1}\frac{k-j}{n-k-j+1}\le \Bigl(\frac 12\Bigr)^{i-2}\le \frac 14.\end{equation*}
$$\frac{{n-i\choose k-i+1}}{{n-i-1\choose k-2}}\le \frac{(n-i)\prod_{j=2}^{i-2}(k-j)}{\prod_{j=1}^{i-2}(n-k-j)}\le \frac{n-4}{n-k-1}\Big(\frac{k-2}{n-k-2}\Big)^{i-3}\le \frac{3k-5}{2k-2}\cdot\frac{1}{2}\le \frac 34.$$
Now we may conclude.
$$(\ref{eq10})\le \Bigl(-\frac 34 k+ \frac {9k}{20}\Bigr){n-i-1\choose k-2}< 0.$$
\end{proof}

\section{Acknowledgements}
The idea to study $m(n,k,l)$, as well as several questions concerning its values for particular values of $n$ and $l$ were communicated to us by A. Raigorodskii. Some of them are answered in this paper. We also thank the anonymous referee for helpful suggestions.



\end{document}